\documentclass[12pt]{article}
\usepackage[T2A]{fontenc}
\usepackage[cp1251]{inputenc}
\usepackage[russian]{babel}
\usepackage{amssymb,amsmath,latexsym,amsthm}
\usepackage{amsbsy,amsfonts,mathrsfs,amscd}
\usepackage{eucal}
\usepackage{graphicx}
\usepackage{indentfirst}
\usepackage{verbatim}
\usepackage{enumerate}
\usepackage{setspace}
\usepackage[pdftex,unicode,colorlinks,linkcolor=blue,
citecolor=red,bookmarksopen,pdfhighlight=/N]{hyperref}

\theoremstyle{plain}
\newtheorem{theorem}{Теорема}
\newtheorem*{theoA}{Теорема A}
\newtheorem*{theoB}{Теорема B}
\newtheorem{lemma}{\bf Лемма}

\newtheorem{corollary}{\bf Следствие}
\theoremstyle{definition}
\newtheorem{definition}{Определение}

\renewcommand{\leq}{\leqslant}
\renewcommand{\geq}{\geqslant}
\newcommand{\dd}{\,{\rm d}}
\def\RR{\mathbb R}
\def\CC{\mathbb C}

\DeclareMathOperator{\clos}{clos}
\DeclareMathOperator{\Int}{int}
\DeclareMathOperator{\Har}{har}
\DeclareMathOperator{\Hol}{Hol}
\DeclareMathOperator{\Zero}{Zero}
\DeclareMathOperator{\sbh}{sbh}
\DeclareMathOperator{\supp}{supp}
\DeclareMathOperator{\comp}{c}
\DeclareMathOperator{\Meas}{Meas}

\begin{document}

\begin{flushleft} {\normalsize УДК 517.53+517.574+517.987.1}\end{flushleft} 

\begin{center}
{\Large \bf К распределению нулевых множеств\\ голоморфных функций. IV. Один критерий\footnote{Исследование выполнено за счёт гранта Российского научного фонда (проект № 18-11-00002).}}
\end{center}
\begin{center}
{\sc Б.~Н.~Хабибуллин, Э.~Б.~Меньшикова}
\end{center}

 В более общем виде доказывается анонсированный ранее результат \cite[теорема 2]{KhaKhaE19}. 

 В соглашениях из \cite{KhaKhaE19}--\cite{KhaRozKha19} всюду 
$\CC_{\infty}:=\CC\cup \{\infty \}$ --- одноточечная компактификация Александрова комплексной плоскости $\CC$,  $S\subset \CC_{\infty}$ --- {\it борелевское подмножество;\/} $\Int S$, $\clos S$  и $\partial S$ ---  {\it внутренность,\/} {\it замыкание\/} и {\it граница\/} $S$ в  $\CC_{\infty}$. 

Далее $\Meas (S)$ --- класс вещественных борелевских мер (Радона) на $S$, или {\it зарядов,\/} $$\Meas^+ (S):=\{\mu \in \Meas (S)\colon \mu\geq 0 \},$$
 $\Meas_{\comp}^+(S)\subset \Meas^+(S) $ --- класс  мер с компактным носителем в $S$,
   $\delta_z$ --- вероятностная {\it мера Дирака с носителем $\supp \delta_z=\{z\}$}.

 Классы $\sbh(S)$, $\Har(S)$,  $\Hol(S)$ состоят  соотв.\footnote{сокращение для <<соответственно>>} из {\it субгармонических,\/}
 {\it гармонических\/} и {\it голоморфных функций\/} в какой-нибудь открытой окрестности  множества  $S\subset \CC_{\infty}$,
\begin{equation}\label{HolDM}
\Hol(D,M):=\Bigl\{f\in \Hol(D)\colon \sup\limits_D {|f|}{\exp(-M)}<+\infty \Bigr\},  
\end{equation}
где  $D\subset \CC_{\infty}$ --- {\it область,\/} т.\,е.  открытое связное подмножество,
 \begin{equation}\label{M}
M:=M_+-M_-\colon D\to \RR\cup\{\pm\infty \}, \quad 
M_{\pm}\in  \sbh(D), \quad M_{\pm}\not\equiv -\infty
\end{equation}
 --- {\it нетривиальная $\delta$-субгармоническая функция\/} \cite[3.1]{KhaRoz18}
с {\it зарядом Рисса\/}  
\begin{equation}\label{nZ}
\nu_M:=\frac1{2\pi}\Delta M\in \Meas (D), \quad \text{$\Delta$ ---\textit{ оператор Лапласа.}} 
\end{equation}

Всюду ${\sf Z}=\{{\sf z}_k\}_{k=1,2, \dots}\subset D$ --- последовательность точек в области $D$ без предельных точек в $D$ со {\it считающей мерой\/} $$n_{\sf Z}:=\sum\limits_{k}\delta_{{\sf z}_k}.$$  

Выберем
\begin{equation}\label{S0S}
 \begin{split}
\text{\it число }b&>0,\text{а также  {\it множества }} \\
 \varnothing\neq \Int S \subset  \clos S  &\subset \Int S_0 \subset \clos S_0\subset D, \text{ \it  где  $S$ или\/ $\Int S_0$ линейно связно}. 
\end{split}
\end{equation} 
Класс {\it вещественных тестовых функций\/}
 $\sbh_0^{\pm}(D\setminus S,S_0;\leq b)$ состоит из субгармонических функций $v\in \sbh(D\setminus S)$,
удовлетворяющих трем условиям \cite[(2)]{KhaKhaE19}:  
\begin{enumerate}[{1)}]
\item {\it $\lim\limits_{D\ni z'\to z}v(z')=0$ для всех\/ $z\in \partial D$;\/} 
\item   \textit{существует $S_v\subset \clos S_v\subset D$, для которого\/ $v\geq 0$  на\/ $D\setminus S_v$}; 
\item $\sup\limits_{S_0\setminus S}|v|\leq b$.
\end{enumerate}
Рассматриваем и {\it финитный вблизи\/} $\partial D$ подкласс  введенного класса тестовых вещественных функций 
 $\sbh_0^{\pm}(D\setminus S,S_0;\leq b)$  (ср. с 
\cite[(1.12)]{KhaKhaF19_III})
$$\sbh_{00}^{\pm}(D\setminus S,S_0;\leq b)
:=\bigl\{v\in \sbh_0^{\pm} (D\setminus S,S_0;\leq b)\colon 
v\equiv 0\text{  \it вне некоторого $S_v\subset \clos S_v\subset D$}\bigr\}.$$ 

\begin{definition}
[{\textit{аффинного выметания\/} \cite[определение 1]{KhaKhaF19_III},
	 \cite[определение 7.1, (7.11)]{KhaRozKha19}}]\label{df1}  Пусть  $S\subset D$, $\mathcal V$
--- некоторый класс \textit{полунепрерывных сверху\/ } функций 
\begin{equation}\label{v}
v\colon D\setminus S\to \{-\infty \}\cup \RR.
\end{equation}
Рассмотрим заряды $\nu,\mu \in \Meas (D)$. Пишем $\nu\curlyeqprec_{S, \mathcal V} \mu$
(соотв.  ${\sf Z}\curlyeqprec_{S,\mathcal V} \mu$), если для некоторого числа  $C\in \RR$ выполнены неравенства 
\begin{equation}\label{bal}
\begin{split}
\int_{D\setminus S}v\dd \nu&\leq \int_{D\setminus S} v \dd \mu+C\\ %%\quad
\biggl(\text{соотв. }\sum_{{\sf z}_k\in D\setminus S} v({\sf z}_k)&\overset{\eqref{nZ}}{:=:}\int_{D\setminus S} v
\dd n_{\sf Z}\leq \int_{D\setminus S} v \dd \mu+C\biggr)
\quad \text{\it для всех $v\overset{\eqref{v}}{\in} \mathcal V$}.
\end{split}
\end{equation}
\end{definition}
\begin{theorem}[{\rm критерий последовательности нулей для $\Hol(D,M)$}]\label{th:cr} Пусть   $D$ односвязная область в $\CC_{\infty}$ и на $\partial D$ более одной точки или $D$ двусвязная  в $\CC_{\infty}$ и на   $\partial D$ более двух точек или  $D$ конечносвязная в $\CC_{\infty}$ с  $\clos D\neq \CC_{\infty}$. Пусть    
функция  $M_+\overset{\eqref{M}}{\in} \sbh (D)$ непрерывна, т.\,е. $M_+\overset{\eqref{M}}{\in} C(D)$. 
Тогда следующие три  утверждения  эквивалентны.
\begin{enumerate}[{\bf z1.}]
\item\label{DM1}
$\sf Z$ --- последовательность нулей для\/ $\Hol(D,M)$, т.\,е.  существует\/ $f\in \Hol(D,M)$ с последовательностью нулей\/ $\Zero_f$, перенумерованной с учётом кратности, со считающей мерой\/  $n_{\Zero_f}\overset{\eqref{nZ}}{=}n_{\sf Z}$  {\rm (пишем  $\Zero_f ={\sf Z}$)}. 
\item\label{DM2} Для любых  объектов из\/ \eqref{S0S}  имеем 
 ${\sf Z}\overset{\eqref{bal}}
 {\curlyeqprec}_{S,\mathcal V_b} \nu_M$ относительно класса 
\begin{equation}\label{Vbz2}
\mathcal V_b{:=} \sbh_0^{\pm}(D\setminus S ,S_0;\leq b).
\end{equation} 

\item\label{DM3} Существуют объекты из \/  \eqref{S0S}, для которых ${\sf Z} \overset{\eqref{bal}}
{\curlyeqprec}_{S,\mathcal V_b} \nu_M$ относительно класса 
\begin{equation}\label{Vbz3}
\mathcal V_b{:=} \sbh_{00}^{\pm}(D\setminus S ,S_0;\leq b)\cap C^{\infty} (D\setminus S ),
\end{equation} 
где $C^{\infty}(D\setminus S)$ --- класс бесконечно дифференцируемых функций
в окрестности $D\setminus S$.
\end{enumerate}
\end{theorem}

Импликация {\bf z\ref{DM1}}$\Rightarrow${\bf z\ref{DM2}} для {\it произвольной области\/} $D\subset \CC^n$ с {\it \underline{субгармонической} функцией\/} 
$M$ доказана в \cite[теорема 1]{KhaKhaE19}. Более общая, чем \cite[теорема 1]{KhaKhaE19},  версия при $n=1$ для  произвольной уже \textit{$\delta$-субгармонической функцией }
  $M\not\equiv \pm \infty $    --- это 
\begin{theorem}\label{th_n}
В условиях  \eqref{S0S} пусть  $u\in \sbh(D)\setminus \{\boldsymbol{-\infty}  \}$ с мерой Рисса $\nu_u:=\frac{1}{2\pi}\Delta u$. 
Если $u\leq M$ на $D$, то $\nu_u\overset{\eqref{bal}}
{\curlyeqprec}_{S,\mathcal V_b} \nu_M$ относительно класса  
$\mathcal V_b\overset{\eqref{Vbz2}}{:=}\sbh_0^{\pm}(D\setminus S ,S_0;\leq b)$.
\end{theorem}
Доказательство теоремы \ref{th_n} опускаем, поскольку оно по сути  повторяет [2, доказательство теорема 1]. Отметим лишь, что предварительно необходимо выбрать точку $z_0\in \Int S$, в которой  $u(z_0)\neq -\infty$ и $M_{\pm}(z_0)\neq -\infty$, а также оперировать вместо неравенства $u\leq M$ неравенством $u+M_-\leq M_+$,  где в обеих частях субгармонические функции.  Роль $\ln |f|$ из  \cite[доказательство теоремы 1]{KhaKhaE19} будет играть функция $u+M_-$ с $M_+$ вместо $M$.    
Тогда $\nu_{u+M_-}\overset{\eqref{bal}}{\curlyeqprec}_{S,\mathcal V_b} \nu_{M_+}$, откуда по определению \ref{df1}  следует 
$\nu_u\overset{\eqref{bal}}{\curlyeqprec}_{S,\mathcal V_b} \nu_M$.

 Из утверждения {\bf z\ref{DM1}} для последовательности ${\sf Z}\subset D$ существует   $f{\in} \Hol (D)$ с $\Zero_{f}={\sf Z}$,  для которой $|f|{\leq} \exp M$ на $D$.  При  $u:=\ln |f|\leq M$ на $D$
имеем $$\nu_u=\frac1{2\pi}\Delta u=\frac{1}{2\pi} \Delta \ln |f| {=}n_{\sf Z}$$ и по теореме  \ref{th_n}  
получаем $n_{\sf Z}\overset{\eqref{bal}}{\curlyeqprec}_{S,\mathcal V_b} \nu_M$, что по определению  \ref{df1} означает ${\sf Z}\overset{\eqref{bal}}{\curlyeqprec}_{S, \mathcal V_b} \nu_M$. Таким образом, импликация   {\bf z\ref{DM1}}$\Rightarrow${\bf z\ref{DM2}} доказана.
	
Импликация  {\bf z\ref{DM2}}$\Rightarrow${\bf z\ref{DM3}}  очевидна,
поскольку класс тестовых функций $\mathcal V_b$ из  {\bf z\ref{DM2}}\eqref{Vbz2} включает  в себя класс тестовых функций $\mathcal V_b$ из 
{\bf z\ref{DM3}}\eqref{Vbz3}. 

Далее для простоты $\infty\notin D$, т.\,е. $D\subset \CC$. Для точки $z\in \CC$ и числа $t>0$ через $D(z,t)\subset \CC$ обозначаем открытый круг с центром $z$ радиуса $t$,
$\overline D(z,t):=\clos D(z,t)$.  Для доказательства импликации   {\bf z\ref{DM3}}$\Rightarrow${\bf z\ref{DM1}} теоремы \ref{th:cr} будет использованы следующие 
\begin{theorem}\label{th3}
Пусть в условиях \eqref{S0S} без линейной связности
граница 
$\partial D$ неполярная, %%\footnote{Это означает, что %%существует функция Грина $g_D$ области $D$.} 
 а также задана мера $\nu\in \Meas^+ (D)$. Если $\nu \overset{\eqref{bal}}
{\curlyeqprec}_{S,\mathcal V_b}\nu_M$ относительно  класса тестовых функций $\mathcal V_b$ из   {\bf z\ref{DM3}}\eqref{Vbz3}, то для любой положительной  функции 
$r\colon D\to \RR$ с ограничениями 
\begin{equation}\label{r}
 \inf_{\overline D(z,t)} r>0 \text{ для любых  $\overline D(z,t)\subset D$}, \quad
\overline D\bigl(z,r(z)\bigr)\subset D\text{ для всех $z\in D$}
 \end{equation}   
найдется   такая функция  $u\in \sbh(D)$ с мерой Рисса $\nu_u=\nu$, что  \begin{equation}\label{*r}
\begin{split}
u(z)&\leq M^{\odot r}(z)\leq M_+^{\odot r}(z)-M_-(z) \quad 
\text{для всех $z\in D$,}\\
\text{где }  \quad M^{\odot r}(z)&:=\frac{1}{2\pi}\int_0^{2\pi}M\bigl(z+r(z)e^{i\theta}\bigr)\dd \theta.
\end{split}
\end{equation} 
Если $M_+\in C(D)$, то функцию $r$ в   \eqref{r}--\eqref{*r} можно выбрать так, что $u
\leq M+1$ на $D$.  
\end{theorem} 

\begin{corollary}\label{cor} Если $D$ --- $(k+1)$-связная область с неполярной   границей $\partial D$, $k<+\infty$,  ${\sf Z}\curlyeqprec_{S,\mathcal V_b}\nu_M$ относительно класса $\mathcal V_b$ из  {\bf z\ref{DM3}}\eqref{Vbz3} и  $M_+\in C(D)$, то найдется   число $a<k$, для которого 
$\sf Z$ --- последовательность нулей для  пространства  $ \Hol \bigl(D,M+a^+\ln^+|\cdot| \bigr)$ вида \eqref{HolDM}, где $a^+:=\max \{ 0,a \}$ и  $\ln^+:=\max \{0,\ln\}$, а  при $\clos D\neq \CC_{\infty}$ 
 имеем $a:=0$. 
\end{corollary}
Для областей $D$ из теоремы  \ref{th_n} по следствию \ref{cor} имеем $a^+=0$ или $a=0$, что сразу доказывает  импликацию {\bf z\ref{DM3}}$\Rightarrow${\bf z\ref{DM1}}
теоремы \ref{th:cr}.

\begin{proof}[Доказательство следствия\/ {\rm \ref{cor}}] По определению \ref{df1} и  условию  ${\sf Z}\curlyeqprec_{S,\mathcal V_b}\nu_M$ с  $M_+\in C(D)$ из теоремы  \ref{th3} следует существование функции $u\in \sbh (D)$ с мерой Рисса $\nu_u=n_{\sf Z}$, удовлетворяющей неравенству $u\leq  M^{*r}\leq M+1$ на $D$ при подходящем выборе функции $r\in C(D)$. Имеет место представление
$u=\ln |f_{\sf Z}|+h$, где $f_{\sf Z}\in \Hol (D)$ с $\Zero_{f_{\sf Z}}={\sf Z}$, а $h\in \Har(D)$. В \cite[лемма 2.1, (2.15)]{Kha07} доказано, что для областей $D$ рассматриваемого в следствии \ref{cor} типа  для  любой функции $h\in \Har (D)$ можно подобрать число $a$ требуемого в следствии \ref{cor} вида, для которого найдется функция $g\in \Hol(D)$ с $\Zero_g=\varnothing$, удовлетворяющая неравенству $\ln \bigl|g(z)\bigr|\leq h(z)+a^+\ln^+|z|$, $z\in D$. Теперь функция 
$f:=f_{\sf Z}g$ с $\Zero_f=\Zero_{f_{\sf Z}}={\sf Z}$ 
удовлетворяет неравенству
$$\ln |f|=\ln |f_{\sf Z}|+\ln |g|\leq \ln |f_{\sf Z}|+h+a^+\ln^+|\cdot|=u+a^+\ln^+|\cdot|\leq M+1+a^+\ln^+|\cdot|$$  на $D$ и ${\sf Z}$ --- последовательность нулей для   $ \Hol \bigl(D,M+a^+\ln^+|\cdot| \bigr)$.  
\end{proof}

\paragraph{Доказательство теоремы\/ { \ref{th3}.}}
%%\begin{proof}[Доказательство теоремы\/ {\rm \ref{th3}}] 
Всегда существует функция $u_{\nu}\in \sbh (D)$ с мерой Рисса $\nu$ \cite{Rans}, \cite{HK}. Не умаляя общности, можем считать, что $0\in \Int S$ и  
$u_{\nu}(0)\neq -\infty$, $M_{\pm}(0)\neq -\infty$. 
\begin{definition}[{\cite[3.1]{Kha07}--\cite[определения 1,2]{Kha03}}]\label{dfAS}
Мера $\mu\in \Meas_{\comp}^+(D)$ называется {\it мерой Аренса\,--\,Зингера\/} для  точки $0\in D$, если $$h(0)=\int h\dd \mu \quad 
\text{для любой функции  $h\in \Har(D)$.}$$ 
Класс всех мер Аренса\,--\,Зингера для  точки $0\in D$ обозначаем через 
$AS_{0}(D)$.  

Функция $V\in \sbh\bigl(D\setminus \{0\}\bigr)$ --- {\it потенциал Аренса\,--\,Зингера для точки\/ $0\in D$ с единичной нормировкой,\/} если  $V\equiv 0$ вне некоторого $S_V\subset \clos S_V\subset D$ и  
\begin{equation}\label{1}
V(z)= -\ln |z|+O(1)\quad\text{при $z\to 0$}.
\end{equation}  
Класс  всех таких потенциалов Аренса\,--\,Зингера для точки\/ $0\in D$ с единичной  нормировкой \eqref{1} обозначаем через  $PAS_{0}^1(D)$. 
\end{definition}

 Через $\Meas_{\infty}^+(D) $ обозначаем   подкласс абсолютно непрерывных мер $\mu\in \Meas^+(D)$ с плотностью $m\in C^{\infty}(D)$, т.\,е. $\dd \mu=m\dd \lambda$,  где   $\lambda$ --- \textit{мера Лебега.}   

\begin{theoA}[{\cite[\S~1, предложения 1.2--4, теорема двойственности]{Kha03}}]\label{l1} Для любых $u\in \sbh(D)$ с $u(0)\neq -\infty$ и с мерой Рисса $\nu_u=\frac1{2\pi}\Delta u$  и потенциала Аренса\,--\,Зингера 
$V\in PAS_{0}^1(D)$ 	 имеет место  расширенная формула Пуассона\,--\,Йенсена
\begin{equation}\label{PJ}
u(0)+\int_{D} V\dd \nu_u=\int_Du\dd \mu_V , \quad \text{где $\mu_V:=\frac1{2\pi}\Delta V\in AS_{0}(D)$}.
\end{equation}	
При этом для любой области $U_0\subset \clos U_0\subset D$ с  $0\in U_0$  оператор  
$$\frac{1}{2\pi}\Delta\colon V\overset{\eqref{PJ}}{\mapsto} \mu_V$$ определяет биекцию
подкласса потенциалов Йенсена
\begin{subequations}\label{bij}
\begin{align}
\mathcal V(U_0)&:= PAS_0^1(D) \cap \Har \bigl(U_0\setminus \{0\}\bigr)\cap C^{\infty}\bigl(D\setminus \{0\}\bigr) 
\tag{\ref{bij}V}\label{bijV}
\\
\intertext{на подкласс мер Йенсена}
\quad  \mathcal M (U_0)&:= AS_0(D)\cap \Meas_{\comp}(D\setminus U_0)\cap \Meas^+_{\infty}(D).
\tag{\ref{bij}M}\label{bijM}
\end{align}
\end{subequations}
\end{theoA}

\begin{lemma}\label{l2} Для любой области\/ $U_0$ при 
$$0\in \Int S\subset \clos S \subset \Int S_0\subset \clos S_0\subset U_0\subset \clos U_0\subset D$$ 
найдётся число $B>0$, для которого имеет место включение
%%найдется число  $B>0$, для которого 
\begin{equation}\label{Vv}
\mathcal V(U_0) \overset{\eqref{bijV}}{\subset} \mathcal V_B:=
\sbh_{00}^{\pm}(D\setminus S ,S_0;\leq B)\cap C^{\infty} (D\setminus S ).
%%\quad\text{для некоторого числа  $B>0$}.
\end{equation}	
\end{lemma}
\begin{proof}[Доказательство леммы\/ {\rm \ref{l2}}]
Для области $D$ с {\it неполярной границей\/} $\partial D$ существует {\it функция Грина\/} 
$g_D$ \cite{Rans},\cite{HK}, и  $V\leq g_D(\cdot , 0)$ для всех $V\in PAS_0^1(D)$
\cite[(2.15)]{KhaKhaF19_III},   откуда, ввиду $0\in \Int S$, получаем {\it оценку сверху\/} $$\sup\limits_{S_0\setminus S} V\leq \sup\limits_{S_0\setminus S} g_D(\cdot,0)=:B',$$ где число $B'>0$ зависит только от $S,S_0,D$. 
Для $V\in PAS_0^1$ имеет место {\it представление\/} 
$$V(z)\overset{\eqref{PJ}}{=}p_{\mu_V}(z)-\ln |z|, \quad z\in D\setminus\{0\},
\quad\text{где } p_{\mu_V}(z):=\int \ln |z-z'| \dd \mu_V(z'),$$
а  $\mu_V$ из \eqref{PJ} --- {\it вероятностная мера\/} \cite[предложение 1.4]{Kha03}.
Выберем число $r>0$ так, что $2r$ --- это минимум двух расстояний: от $0$ до $\partial S$ и 
от $S_0$ до $\partial U_0$. В \cite[предложение 2.6]{KhaBai16} для 
усреднений по кругам
$$p^{\bullet r}_{\mu_V}(z):=\frac{1}{\pi r^2}\int_{\overline{D}(z,r)} p_{\mu_V}(z')\dd \lambda (z')
$$ 
получена оценка снизу $$p^{\bullet r}_{\mu_V}(z)\geq \mu_V(\CC)\ln ({r}/{\sqrt{e}})=
\ln ({r}/{\sqrt{e}})=:B''\quad \text{для всех $z\in \CC$}. 
$$ Но  согласно выбору $r$ и ввиду $\supp \mu_V\overset{\eqref{bijM}}{\subset} D\setminus U_0$ из гармоничности  потенциала $p_{\mu_V}$ вне $\supp \mu_V$ \cite{Rans}, \cite{HK} имеем 
$$p_{\mu_V}(z)=p^{\bullet r}_{\mu_V}(z)\geq B''\quad\text{для всех $z\in S_0\setminus S$}.$$ 
Отсюда $$V(z)\geq B''+\inf_{z'\in S_0\setminus S}  \ln (1/{|z'|})=:B'''
\quad\text{для всех $z\in S_0\setminus S$},$$ где по построению  $B'''\in \RR$  зависит только от $r, S,S_0,U_0$, а стало быть лишь от взаимного расположения точки $0$ и  множеств $S,S_0,U_0$. Таким образом, в соответствии с определением  \ref{dfAS} получаем  включение  \eqref{Vv} с $B:=\max\{B',|B'''|\}>0$. 
\end{proof}

По определению \ref{df1} для произвольно большого числа $B>0$,  
 домножив обе части неравенства \eqref{bal} на $B/b>0$,  в теореме \ref{th3}
 класс $\mathcal V_b$  можно заменить на класс  $\mathcal V_B$ из \eqref{Vv} и по лемме \ref{l2} из условия $\nu \overset{\eqref{bal}}{\curlyeqprec}_{S,\mathcal V_B} \nu_M$ 
 следует $\nu \overset{\eqref{Vv}}{\curlyeqprec}_{S,\mathcal V(U_0)} \nu_M$.
Последнее по определению \ref{df1} аффинного выметания в \eqref{bal}, применяя  расширенную формулу Пуассона\,--\,Йенсена \eqref{PJ} к $u_\nu$ и $M_\pm$, 
а также  биекцию \eqref{bij}, можно переписать  в виде
\begin{equation}\label{Mu}
\int u_\nu \dd \mu\leq \int M \dd \mu +\underbrace{(BC/b+u(0)-M_+(0)+M_-(0))}_{c}\quad 
\text{\it для всех $\mu \overset{\eqref{bij}}{\in} \mathcal M(U_0)$},
\end{equation}
  где $c\in \RR$ --- некоторая постоянная.
 
\begin{theoB}[{\rm очень частный случай \cite[следствие 8.1]{KhaRozKha19} при $H=\Har(D)$}]
Если для  некоторого числа $c\in \RR$ выполнено  \eqref{Mu}, то для любой функции $r$, удовлетворяющей условиям \eqref{r}, найдутся такие функция  $h\in \Har(D)$
и   положительная функция $\widehat{r}\leq r$ из класса  $C^{\infty}(D)$, что имеет место неравенство 
\begin{equation}\label{uM}
u_\nu+h\leq M^{\circledast\widehat{r}}\in C^{\infty}(D)\quad\text{на $D$}, 
\end{equation}
где, по построению  \/ {\rm  \cite[(8.3--6), (8.10)]{KhaRozKha19},  \cite[(2.18--19)]{KhaKhaF19_III}}, 
$M^{\circledast\widehat{r}}(z)$ --- <<скользящие сжимающиеся>> сглаживающие  усреднения по некоторым вероятностным мерам 
$$\alpha^{(\widehat r(z))}\in \Meas^+_{\infty}\bigl(\overline D(z,\widehat r(z))\bigr),$$
 полученным сдвигом, сжатием и нормировкой некоторой единой  аппроксимативной единицы $a\in C^\infty(\CC)$, зависящей только от модуля $|\cdot|$, с  
носителем $\supp a\subset \overline{D}(0,1)$.  
\end{theoB}
По  теореме B требуемой  функцией выбираем  
\begin{equation}\label{Munu}
u\overset{\eqref{uM}}{:=}u_\nu+h\overset{\eqref{uM}}{\leq} M^{\circledast\widehat{r}}=M_+^{\circledast\widehat{r}}-M_-^{\circledast\widehat{r}}\leq 
M_+^{\circledast\widehat{r}}-M_-\quad \text{на $D$},
\end{equation}
 поскольку для $M_-\in \sbh (D)$ имеем 
$M_-\leq M_-^{\circledast\widehat{r}}$ на $D$.  Из $M_+\in \sbh (D)$ следует 
$$
M_+^{\circledast\widehat{r}}\overset{\eqref{*r}}{\leq} M_+^{\odot\widehat{r}}\leq 
M_+^{\odot{r}}\quad \text{на $D$},$$
что дает неравенство $u\leq M_+^{\odot{r}}-M_-$ на $D$. Если $M_+\in C(D)$, 
то в силу локально равномерной непрерывности функции $M_+$ функцию $r$ в ограничениях  \eqref{r} можно выбрать столь малой, что $M_+^{\odot{r}} \leq M+1$ на $D$. Теорема \ref{th3} доказана.

\renewcommand{\refname}{\begin{center}{\normalsize \sc Литература}\end{center}}

Башкирский государственный университет

e-mail: khabib-bulat@mail.ru 

%%%e-mail: khabibullinae@gmail.com

\end{document}